\documentclass{amsart}

\usepackage[latin1]{inputenc}
\usepackage[english]{babel}
\usepackage{indentfirst}
\usepackage{amssymb}
\usepackage{amsthm}
\usepackage{xcolor}
\usepackage[all]{xy}

\newcommand{\N}{\mathbb{N}}

\newcommand{\sub}{\subseteq}
\def\epsilon{\varepsilon}
\newtheorem{theo}{Theorem}
\newtheorem{lem}[theo]{Lemma}
\newtheorem{cor}[theo]{Corollary}
\newtheorem{defi}[theo]{Definition}

\numberwithin{equation}{section}

\title{On the range of a vector measure}

\author{Jos\'e Rodr\'{i}guez}
\address{Dpto. de Ingenier\'{i}a y Tecnolog\'{i}a de Computadores,
Facultad de Inform\'{a}tica, Universidad de Murcia, 30100 Espinardo (Murcia), Spain}
\email{joserr@um.es}



\thanks{Research supported by projects MTM2017-86182-P (AEI/FEDER, UE)
and 20797/PI/18 (Fundaci\'on S\'eneca).}

\begin{document}

\maketitle

\begin{abstract}
Let $(\Omega,\Sigma,\mu)$ be a finite measure space, $Z$ be a Banach space
and $\nu:\Sigma \to Z^*$ be a countably additive $\mu$-continuous vector measure.
Let $X \sub Z^*$ be a norm-closed subspace which is norming for~$Z$. 
Write $\sigma(Z,X)$ (resp. $\mu(X,Z)$) to denote the weak (resp. Mackey) topology on~$Z$ (resp.~$X$)
associated to the dual pair $\langle X,Z\rangle$. 
Suppose that, either $(Z,\sigma(Z,X))$ has the Mazur property, or $(B_{X^*},w^*)$ is convex block compact
and $(X,\mu(X,Z))$ is complete. 
We prove that the range of~$\nu$ is contained in~$X$ if, for each $A\in \Sigma$ with $\mu(A)>0$,
the $w^*$-closed convex hull of $\{\frac{\nu(B)}{\mu(B)}: \, B\in \Sigma, \, B \sub A, \, \mu(B)>0\}$
intersects~$X$. This extends results obtained by Freniche [Proc. Amer. Math. Soc. 107 (1989), no.~1, 119--124]
when $Z=X^*$.
\end{abstract}

\section{Introduction}

Throughout this paper $(\Omega,\Sigma,\mu)$ is a finite measure space, $Z$
is a (real) Banach space and $X\sub Z^*$ is a norm-closed subspace which is {\em norming} for~$Z$,
meaning that the formula $|||z|||=\sup\{\langle x,z \rangle:x\in B_X\}$ defines
an equivalent norm on~$Z$. As usual, $B_X$ denotes the closed unit ball of~$X$ and
the evaluation of $z^*\in Z^*$ at $z\in Z$ is denoted by $\langle z^*,z\rangle$. 
The linear map $r: Z\to X^*$ defined by 
$$
	r(z)(x):=\langle x,z \rangle 
	\quad\mbox{for all }
	z\in Z \mbox{ and }x\in X
$$ 
is an isomorphic embedding and $r(Z)\sub X^*$ is norming for~$X$. Note that $r$ is a homeomorphism
between $(Z,\sigma(Z,X))$ and $(r(Z),w^*)$, where $\sigma(Z,X)$ is the topology on~$Z$ of pointwise convergence on~$X$
and $w^*$ is the weak$^*$-topology.

Given a countably additive $\mu$-continuous vector measure $\nu:\Sigma\to Z^*$
(shortly $\nu \in ca(\mu,Z^*)$), we study whether its range $\nu(\Sigma)=\{\nu(A):A\in \Sigma\}$ is contained in~$X$
provided that
\begin{equation}\label{eqn:GG}
	\overline{{\rm co}(\mathcal{A}_\nu(A))}^{w^*}\cap X \neq \emptyset
	\quad
	\mbox{for every }A\in \Sigma \mbox{ with }\mu(A)>0.\tag{G}
\end{equation}
Here ${\rm co}(\mathcal{A}_\nu(A))$ denotes the convex hull of the ``average range''
$$
	\mathcal{A}_\nu(A):=\Big\{\frac{\nu(B)}{\mu(B)}: \, B\in \Sigma, \, B \sub A, \, \mu(B)>0\Big\}.
$$

A particular case of remarkable interest arises when $Z=X^{*}$ and $\nu$ is the indefinite
Dunford integral of a scalarly measurable and scalarly bounded function $f:\Omega \to X$. 
In this case, the celebrated Geitz-Talagrand ``core'' theorem (see~\cite{gei-2} and \cite[5-2-2]{tal}) ensures 
that $\nu(\Sigma) \sub X$ (i.e. $f$ is Pettis integrable) whenever 
condition~\eqref{eqn:GG} holds. Motivated by this result, Freniche~\cite{fren} discussed the question of whether condition~\eqref{eqn:GG}
implies the inclusion $\nu(\Sigma)\sub X$ for an arbitrary $\nu\in ca(\mu,X^{**})$. While this question
remains open in full generality, he proved that the answer is affirmative under each of the following assumptions on~$X$:
\begin{itemize}
\item[(a)] every $w^*$-sequentially continuous linear functional on~$X^*$ is $w^*$-continuous, i.e.
$(X^*,w^*)$ has the Mazur property;
\item[(b)] every sequence in~$B_{X^*}$ admits a $w^*$-convergent subsequence, i.e.
 $(B_{X^*},w^*$) is sequentially compact.
\end{itemize}
Note that both (a) and~(b) hold if $X$ is weakly compactly generated and, more generally,
if $(B_{X^*},w^*)$ is Fr\'{e}chet-Urysohn (meaning that the $w^*$-closure of any 
set $C\sub B_{X^*}$ consists of limits of $w^*$-convergent sequences contained in~$C$).

In this paper we push a bit further Freniche's techniques to obtain generalizations
of his results above.
Our discussion involves the Mazur property and the completeness of the Mackey topology
of dual pairs associated to norming subspaces; these topics have been studied recently in 
\cite{bon-cas,gui-mar-rod,gui-mon,gui-mon-ziz}. 
Recall that a locally convex space $E$ is said to have the {\em Mazur property}
if every sequentially continuous linear functional from~$E$ to~$\mathbb{R}$ is continuous.
Our first main result reads as follows:

\begin{theo}\label{theo:MainA}
Suppose that $(Z,\sigma(Z,X))$ has the Mazur property.
If $\nu\in ca(\mu,Z^*)$ satisfies condition~\eqref{eqn:GG}, then $\nu(\Sigma) \sub X$.
\end{theo}
 
The Banach space $X$ is said to have {\em Efremov's property~($\mathcal{E}$)} 
if the $w^*$-closure of any convex set $C\sub B_{X^*}$ consists of limits of $w^*$-convergent
sequences contained in~$C$. Obviously, this property holds if $(B_{X^*},w^*)$ is Fr\'{e}chet-Urysohn. 
Under the Continuum Hypothesis there exist Banach spaces separating both properties (see~\cite{avi-mar-rod}), but it
is unknown what happens in general. The relevance of property~($\mathcal{E}$) to our discussion stems 
from the fact that if $X$ has property~($\mathcal{E}$), then $(r(Z),w^*)$ has the Mazur property (see \cite[Corollary~3.4]{gui-mar-rod})
and so does $(Z,\sigma(Z,X))$. As a consequence:

\begin{cor}\label{cor:Efremov}
Suppose that $X$ has Efremov's property~($\mathcal{E}$). 
If $\nu\in ca(\mu,Z^*)$ satisfies condition~\eqref{eqn:GG}, then $\nu(\Sigma) \sub X$.
\end{cor}

Our second main result is:

\begin{theo}\label{theo:MainB}
Suppose that $(B_{X^*},w^*)$ is convex block compact and that $(X,\mu(X,Z))$ is complete.
If $\nu\in ca(\mu,Z^*)$ satisfies condition~\eqref{eqn:GG}, then $\nu(\Sigma) \sub X$.
\end{theo}

Here $\mu(X,Z)$ is the Mackey topology on~$X$ associated to the dual 
pair~$\langle X,Z \rangle$, i.e. the topology on~$X$ of uniform convergence on all
absolutely convex $\sigma(Z,X)$-compact subsets of~$Z$.
According to a result of Grothendieck 
(see e.g. \cite[\S21.9]{kot}), $(X,\mu(X,Z))$ is complete if, and only if,
the $\sigma(Z,X)$-continuity of any linear functional $\varphi:Z\to \mathbb{R}$ 
(i.e. the fact that $\varphi\in X$) is equivalent to the $\sigma(Z,X)$-continuity of the restriction~$\varphi|_K$ for every absolutely convex $\sigma(Z,X)$-compact set $K \sub Z$. 
We stress that the completeness of $(X,\mu(X,Z))$ is weaker
than the Mazur property of $(Z,\sigma(Z,X))$, see \cite[Proposition~10]{gui-mon-ziz}.

Recall that $(B_{X^*},w^*)$ is said to be {\em convex block compact}
if every sequence in~$B_{X^*}$ admits a convex block subsequence which is $w^*$-convergent.
By a {\em convex block subsequence} 
of a sequence $(g_n)_{n\in \N}$ in a linear space we mean a sequence $(h_k)_{k\in \N}$ of the form 
$$
	h_k=\sum_{n\in I_k}a_n g_n,
$$
where $(I_k)_{k\in \N}$ is a sequence of finite subsets of~$\N$ with $\max(I_k) < \min(I_{k+1})$ and $(a_n)_{n\in \N}$ is a sequence
of non-negative real numbers such that $\sum_{n\in I_k}a_n=1$ for all~$k \in \N$.
Convex block compactness is strictly weaker than sequential compactness. Indeed, a result of Bourgain
states that $(B_{X^*},w^*)$ is convex block compact whenever $X$ contains no isomorphic copy of~$\ell_1$, 
see \cite[Proposition~3.11]{bou:79} (cf. \cite{sch-3} and~\cite[Proposition~11]{pfi-J}), while
there exist Banach spaces not containing isomorphic copies of~$\ell_1$ whose
dual ball is not $w^*$-sequentially compact, see \cite{hag-ode,hay10}. At this point it is worth mentioning
that $(B_{X^*},w^*)$ is convex block compact whenever $X$ has Efremov's property~$(\mathcal{E})$ (see \cite[Theorem~3.2.11]{gon4})
or $X=C(K)$ for a compact space~$K$ such that all Radon probabilities on~$K$ have countable type (see \cite[3B]{hay-lev-ode}, cf. \cite{kru-ple}).
It is easy to see that the latter implies Bourgain's result, since every Radon probability on~$(B_{X^*},w^*)$ has countable type 
if $X$ contains no isomorphic copy of~$\ell_1$, see \cite[Proposition~B.1]{avi-mar-ple}.

The Banach-Dieudonn\'{e} theorem and Grothendieck's aforementioned result imply that $(X,\mu(X,X^*))$ is complete for any Banach space~$X$. 
Therefore, as a particular case of Theorem~\ref{theo:MainB} we get the following improvement of
\cite[Theorem~2]{fren}:

\begin{cor}\label{cor:Bidual}
Suppose that $(B_{X^*},w^*)$ is convex block compact.
If $\nu\in ca(\mu,X^{**})$ satisfies condition~\eqref{eqn:GG}, then $\nu(\Sigma) \sub X$.
\end{cor}

The proofs of Theorems~\ref{theo:MainA} and~\ref{theo:MainB} are included in the next section. We
finish this introduction by exhibiting some examples of couples $(X,Z)$ for which our results might be applied (besides
that of $Z=X^*$): 
\begin{itemize}
\item Let $Z$ be a non-reflexive Banach space and pick any $z^{**}\in Z^{**}\setminus Z$. Then
$X=\ker(z^{**}) \sub Z^*$ is norming for~$Z$ (cf. \cite[Lemma~11]{gui-mon-ziz}).
\item Let $Z=\ell_1(\Gamma)$ for a non-empty set~$\Gamma$. Then any
norm-closed subspace $X \sub Z^*=\ell_\infty(\Gamma)$
containing $c_0(\Gamma)$ is norming for~$Z$. 
\item Let $Z=\ell_1(K)$ for a compact space~$K$. Then any norm-closed subspace $X \sub Z^*=\ell_\infty(K)$
containing $C(K)$ is norming for~$Z$.
\end{itemize}

\section{Proofs of Theorems~\ref{theo:MainA} and~\ref{theo:MainB}}

Recall that a set $\mathcal{F}\sub L_1(\mu)$ is called {\em uniformly integrable} if it is norm-bounded
and for every $\epsilon>0$ there is $\delta>0$ such that $\sup_{f\in \mathcal{F}}\int_A |f| \, d\mu\leq \epsilon$
for every $A\in \Sigma$ with $\mu(A)\leq\delta$. This is equivalent to being relatively weakly compact in~$L_1(\mu)$,
see e.g. \cite[p.~76, Theorem~15]{die-uhl-J}.

\begin{defi}
Given $\nu \in ca(\mu,Z^*)$ and $z\in Z$, we denote by $\langle \nu,z\rangle \in ca(\mu,\mathbb{R})$
the composition of $\nu$ and~$z$ (i.e $\langle \nu,z\rangle(A):=\langle \nu(A),z\rangle$ for all $A\in \Sigma$), and we write 
$$
	f_{z}:=\frac{d\langle \nu,z\rangle}{d\mu}
$$
to denote its Radon-Nikod\'{y}m derivative with respect to~$\mu$.
\end{defi}

\begin{lem}\label{lem:RN-UI} 
Let $\nu \in ca(\mu,Z^*)$. If $C \sub Z$ is norm-bounded, then $\{f_{z}: z\in C\}$
is a uniformly integrable subset of~$L_1(\mu)$.
\end{lem}
\begin{proof} Write $M:=\sup_{z\in C}\|z\|$. For every $A\in \Sigma$ and every $z\in C$ we have
\begin{equation}\label{eqn:UI}
	\int_A|f_z| \, d\mu=|\langle \nu,z\rangle|(A) \leq 2 \sup_{\substack{B \in \Sigma \\ B\sub A}}|\langle \nu(B),z \rangle|
	\leq 2M \sup_{\substack{B \in \Sigma \\ B\sub A}}\|\nu(B)\|,
\end{equation}
where $|\langle \nu,z\rangle|$ is the variation of $\langle \nu,z\rangle$. Since $\nu$ has norm-bounded range 
and is $\mu$-continuous,
the uniform integrability of $\{f_{z}: z\in C\}$ follows from~\eqref{eqn:UI}.
\end{proof}

From now on we write $\Sigma^+:=\{A\in \Sigma:\mu(A)>0\}$
and $\Sigma_A^+:=\{B\in \Sigma^+: B \sub A\}$ for all $A\in \Sigma^+$.

\begin{lem}\label{lem:EquiIntegrable}
Let $(g_n)_{n\in \N}$ be a uniformly integrable sequence in~$L_1(\mu)$ for which there exist $\epsilon>0$ and $B\in \Sigma^+$
such that $\int_B g_n\, d\mu\geq \epsilon$ for all $n\in \N$. Then there exist a convex block subsequence 
$(h_k)_{k\in \N}$ of~$(g_n)_{n\in \N}$, $\eta>0$ and $A\in \Sigma_B^+$ such that $h_k\geq \eta$ on~$A$ for all $k\in \N$. 
\end{lem}
\begin{proof}
Since $(g_n)_{n\in \N}$ is relatively weakly compact in~$L_1(\mu)$, it admits weakly convergent subsequences, by the Eberlein-Smulyan theorem. 
Thus, Mazur's theorem applied to any weakly convergent subsequence ensures 
the existence of a convex block subsequence 
$(h_k)_{k\in \N}$ of~$(g_n)_{n\in \N}$ which converges in norm to some $h\in L_1(\mu)$.
By passing to a further subsequence of~$(h_k)_{k\in \N}$, not relabeled, we can assume that $(h_k)_{k\in \N}$ also converges to~$h$ $\mu$-a.e.

Note that $\int_B h_k\, d\mu\geq \epsilon$ for all $k\in \N$, hence $\int_B h\, d\mu\geq \epsilon$ and so 
there exist $\eta'>0$ and $A'\in \Sigma_B^+$ such that $h\geq \eta'$ on~$A'$.
By Egorov's theorem, there is $A\in \Sigma_{A'}^+$ such that $(h_k)_{k\in \N}$ converges to~$h$ uniformly on~$A$.
Take any $0<\eta<\eta'$. Then there is $k_0\in \N$ such that $h_k\geq \eta$ on~$A$ for all $k \geq k_0$. Thus, $(h_k)_{k\geq k_0}$
is a convex block subsequence of $(g_n)_{n\in \N}$ satisfying the required property.
\end{proof}

We say that a subset of a locally convex space is {\em relatively convex block compact}
if every sequence in it admits a convergent convex block subsequence.

\begin{lem}\label{lem:bdd}
A subset of~$Z$ is norm-bounded if it is either relatively compact or relatively convex block compact in~$(Z,\sigma(Z,X))$.
\end{lem}
\begin{proof} Let $S \sub Z$.
If $S$ is relatively compact in $(Z,\sigma(Z,X))$, then $\overline{r(S)}^{w^*}\sub X^*$ is $w^*$-compact, hence $r(S)$ is norm-bounded and so is~$S$. 

If $S$ is relatively convex block compact in~$(Z,\sigma(Z,X))$, then 
$r(S)$ is relatively convex block compact in~$(X^*,w^*)$. This 
implies that $r(S)$ is norm-bounded, by the Banach-Steinhaus theorem
and the fact that every relatively convex block compact subset of~$\mathbb{R}$ is bounded.
It follows that $S$ is norm-bounded.
\end{proof}

It is clear that Theorem~\ref{theo:MainA} follows immediately from the following
generalization of \cite[Theorem~1]{fren}:
 
\begin{theo}\label{theo:SequentialContinuity}
If $\nu\in ca(\mu,Z^*)$ satisfies condition~\eqref{eqn:GG}, then $\nu(B)$
is $\sigma(Z,X)$-sequentially continuous for every $B\in \Sigma$.
\end{theo}
\begin{proof}
Suppose that $\nu(B)$ is not $\sigma(Z,X)$-sequentially continuous for some $B\in \Sigma$. Then there exist a $\sigma(Z,X)$-null sequence $(z_n)_{n\in \N}$ in~$Z$
and $\epsilon>0$ such that 
$$
	\int_B f_{z_n} \, d\mu = \langle \nu(B),z_n\rangle \geq \epsilon \quad
	\mbox{for all }n\in \N.
$$
Since $(z_n)_{n\in \N}$ is relatively compact in $(Z,\sigma(Z,X))$, it is norm-bounded (Lemma~\ref{lem:bdd}), 
so the sequence $(f_{z_n})_{n\in \N}$ is uniformly integrable
(Lemma~\ref{lem:RN-UI}). We can now apply Lemma~\ref{lem:EquiIntegrable} to find a convex block subsequence 
$(h_k)_{k\in \N}$ of~$(f_{z_n})_{n\in \N}$, $\eta>0$ and $A\in \Sigma_B^+$ such that 
\begin{equation}\label{eqn:CBS}
	h_k \geq \eta \, \text{ on $A$ for all $k\in \N$}.
\end{equation}
Clearly, $h_k=f_{\tilde{z}_k}$ for some convex block subsequence
$(\tilde{z}_k)_{k\in \N}$ of~$(z_n)_{n\in \N}$.

Given any $k\in \N$, we have 
$$
	\Big\langle \frac{\nu(C)}{\mu(C)},\tilde{z}_k \Big\rangle=\frac{1}{\mu(C)}\int_C f_{\tilde{z}_k} \, d\mu \stackrel{\eqref{eqn:CBS}}{\geq} \eta
	\quad\mbox{for every }C\in \Sigma_A^+,
$$
and therefore
\begin{equation}\label{eqn:CBS2}
	\langle z^*,\tilde{z}_k \rangle \geq \eta
	\quad \mbox{for every }z^*\in \overline{{\rm co}(\mathcal{A}_\nu(A))}^{w^*}.
\end{equation}
Since $(\tilde{z}_k)_{k\in \N}$ is $\sigma(Z,X)$-null
(because it is a convex block subsequence of the $\sigma(Z,X)$-null sequence $(z_n)_{n\in \N}$), 
from~\eqref{eqn:CBS2} it follows that 
$$
	\overline{{\rm co}(\mathcal{A}_\nu(A))}^{w^*}\cap X= \emptyset,
$$
which contradicts condition~\eqref{eqn:GG}.
\end{proof}

The following result is the key to prove Theorem~\ref{theo:MainB}.

\begin{theo}\label{theo:ConvexBlockCompactLocal}
Let $K \sub Z$ be a set such that ${\rm co}(K)$ is relatively convex block compact in~$(Z,\sigma(Z,X))$.
If $\nu\in ca(\mu,Z^*)$ satisfies condition~\eqref{eqn:GG}, 
then the restriction $\nu(B)|_K$ is $\sigma(Z,X)$-continuous for every $B\in \Sigma$.  
\end{theo}
\begin{proof}
Note first that $K$ is norm-bounded by Lemma~\ref{lem:bdd}.
Suppose, by contradiction, that $\nu(B)|_K$ is not $\sigma(Z,X)$-continuous for some $B\in \Sigma$.
Since $\nu(B)|_K$ is bounded, there is a net $(z_\alpha)$ in~$K$
which $\sigma(Z,X)$-converges to some $z\in K$ and such that the net $(\langle \nu(B),z_\alpha\rangle)$
converges to $\lambda\in \mathbb{R}$ with $\lambda\neq \langle \nu(B),z\rangle$.
By linearity, we can assume that $z=0$ and $\lambda>0$. Now, we can also assume that for some $\epsilon>0$ we have
$$
	\int_B f_{z_\alpha} \, d\mu = \langle \nu(B),z_\alpha\rangle \geq \epsilon \quad
	\mbox{for all }\alpha.
$$ 
The net $(f_{z_\alpha})$ is uniformly integrable (apply Lemma~\ref{lem:RN-UI}) 
and by passing to a further subnet, 
not relabeled, we can suppose that $(f_{z_\alpha})$
is weakly convergent to some $f\in L_1(\mu)$. In particular, $\int_B f \, d\mu \geq \epsilon$.
Take $\eta>0$ and $A\in \Sigma_B^+$ such that 
\begin{equation}\label{eqn:big}
	f\geq \eta \, \text{ on~$A$}. 
\end{equation}
We will contradict condition~\eqref{eqn:GG} by proving the following claim.

{\em Claim:} $\overline{{\rm co}(\mathcal{A}_\nu(A))}^{w^*}\cap X=\emptyset$. Indeed, fix $x\in X$. Take any $n\in\N$. Then there is 
$\alpha_n$ such that 
$$
	|\langle x,z_\alpha \rangle|\leq \frac{1}{n}
	\quad\mbox{for all }\alpha\geq \alpha_n.
$$
By Mazur's theorem, we
can find $g_n\in {\rm co}\{f_{z_\alpha}:\alpha\geq \alpha_n\}$ such that
\begin{equation}\label{eqn:estrella}
	\|g_n-f\|_{L_1(\mu)}\leq \frac{1}{n}.
\end{equation}
Clearly, $g_n=f_{\tilde{z}_n}$ for some
$\tilde{z}_n\in {\rm co}\{z_\alpha:\alpha \geq \alpha_n\} \sub {\rm co}(K)$ and we have 
\begin{equation}\label{eqn:x}
	|\langle x,\tilde{z}_n \rangle|\leq \frac{1}{n}. 
\end{equation}

Since ${\rm co}(K)$ is relatively convex block compact in $(Z,\sigma(Z,X))$, there is a
convex block subsequence $(w_k)_{k\in \N}$ of~$(\tilde{z}_n)_{n\in \N}$
which $\sigma(Z,X)$-converges to some $w\in Z$. 
Note that $(\langle z^*,w_k \rangle)_{k\in\N}$ is a convex block subsequence
of $(\langle z^*,\tilde{z}_n\rangle)_{n\in \N}$ for every $z^*\in Z^*$.
By~\eqref{eqn:x} we have 
$$
	\langle x,w \rangle=\lim_{k\to \infty}\langle x,w_k\rangle=\lim_{n\to \infty}\langle x,\tilde{z}_n\rangle=0.
$$ 

Now, in order to show
that $x\not\in \overline{{\rm co}(\mathcal{A}_\nu(A))}^{w^*}$ we will check that $\langle z^*,w\rangle \geq \eta$ 
for every $z^*\in \overline{{\rm co}(\mathcal{A}_\nu(A))}^{w^*}$. Given any 
$C\in \Sigma_A^+$, inequalities~\eqref{eqn:big} and~\eqref{eqn:estrella} yield
$$
	\lim_{n\to\infty}\langle \nu(C),\tilde{z}_n\rangle =
	\lim_{n\to\infty} \int_C f_{\tilde{z}_n}\, d\mu=\int_C f \, d\mu \geq \eta \mu(C).
$$
On the other hand, the $\sigma(Z,X)$-sequential continuity of $\nu(C)$ (Theorem~\ref{theo:SequentialContinuity}) implies that 
$$
	\langle \nu(C),w \rangle=\lim_{k\to \infty} \langle \nu(C),w_k \rangle=\lim_{n\to\infty}\langle \nu(C),\tilde{z}_n\rangle.
$$
Hence $\langle \frac{\nu(C)}{\mu(C)},w \rangle\geq \eta$. This shows that $\langle z^*,w \rangle\geq \eta$ 
for every $z^*\in \overline{{\rm co}(\mathcal{A}_\nu(A))}^{w^*}$ and the {\em Claim} is proved.
\end{proof}

\begin{proof}[Proof of Theorem~\ref{theo:MainB}]
Fix $B\in \Sigma$. Since $(X,\mu(X,Z))$ is complete, in order to check that $\nu(B)\in X$ it suffices 
to show that the restriction $\nu(B)|_K$ is $\sigma(Z,X)$-continuous for each absolutely convex $\sigma(Z,X)$-compact set~$K \sub Z$.
This follows from Theorem~\ref{theo:ConvexBlockCompactLocal}, because
$K$ is relatively convex block compact in $(Z,\sigma(Z,X))$. Indeed,
let $(z_n)_{n\in\N}$ be a sequence in~$K$. Since $(r(z_n))_{n\in\N}$ is norm-bounded
(Lemma~\ref{lem:bdd}) and $(B_{X^*},w^*)$ is convex block compact, there is 
a convex block subsequence $(\tilde{z}_k)_{k\in \N}$ of~$(z_n)_{n\in\N}$
such that $(r(\tilde{z}_k))_{k\in \N}$ is $w^*$-convergent to some $x^*\in X^*$.
Bearing in mind that $K$ is convex and $r(K)$ is $w^*$-closed (it is $w^*$-compact), we have $x^*=r(z)$ for some $z\in K$. 
Therefore, $(\tilde{z}_k)_{k\in \N}$ is $\sigma(Z,X)$-convergent to~$z$. 
\end{proof}

Following~\cite{gui-mar-rod}, the Banach space $X$ is called {\em fully Mackey complete} 
if $(X,\mu(X,Y))$ is complete for any norm-closed subspace $Y\sub X^*$ which is norming for~$X$. 
Every Banach space having Efremov's property~($\mathcal{E}$)
is fully Mackey complete (see~\cite{gui-mar-rod}).
Thus, the next result (obtained under the set theoretic assumption that ``$\mathfrak{p}>\omega_1$'')
generalizes Corollary~\ref{cor:Efremov}.

\begin{cor}\label{cor:fully}
Suppose that $\mathfrak{p}>\omega_1$ and that $X$ is fully Mackey complete. 
If $\nu\in ca(\mu,Z^*)$ satisfies condition~\eqref{eqn:GG}, then $\nu(\Sigma) \sub X$.
\end{cor}

Before the proof of Corollary~\ref{cor:fully}, recall that~$\mathfrak{p}$ is the least cardinality of a family $\mathcal{M}$ of infinite
subsets of~$\N$ such that every finite subfamily of~$\mathcal{M}$ has infinite intersection, but there is no infinite set 
$A \sub \N$ such that $A \setminus M$ is finite for all $M\in \mathcal{M}$.
In general, $\mathfrak{p}$ lies between $\omega_1$ (the first uncountable ordinal) and $\mathfrak{c}$ (the cardinal of the continuum).
Martin's Axiom implies that $\mathfrak{p}=\mathfrak{c}$, so one has $\mathfrak{p}>\omega_1$ subject
to Martin's Axiom and the negation of the Continuum Hypothesis. We refer the reader to~\cite{bla-J} for 
detailed information on cardinal~$\mathfrak{p}$. 

\begin{proof}[Proof of Corollary~\ref{cor:fully}] Any fully Mackey complete Banach space cannot contain isomorphic copies of~$\ell_1(\omega_1)$, see
\cite[Corollary~4.3]{gui-mar-rod}. Therefore, $X$ contains no isomorphic copy of~$\ell_1(\mathfrak{p})$ which, under the assumption
that $\mathfrak{p}>\omega_1$, implies that $(B_{X^*},w^*)$ is convex block compact, see \cite[3D]{hay-lev-ode}.
The conclusion follows from Theorem~\ref{theo:MainB}, bearing in mind that the norm-closed
subspace $r(Z) \sub X^*$ is norming for~$X$ and so $(X,\mu(X,Z))=(X,\mu(X,r(Z)))$ is complete.
\end{proof}

We stress that the absence of isomorphic copies of~$\ell_1(\mathfrak{c})$
is necessary for the convex block compactness of $(B_{X^*},w^*)$, see \cite[p.~269, Remark~3]{mor-J}.

\subsection*{Acknowledgements}
The author wishes to thank A.J. Guirao for helpful suggestions.
Research supported by projects 
MTM2017-86182-P (AEI/FEDER, UE) and 19275/PI/14 (Fundaci\'on S\'eneca).

\bibliographystyle{amsplain}

\end{document}